\documentclass[12pt]{amsart}
\usepackage{amsmath}
\usepackage{amssymb}
\usepackage{amsfonts}
\usepackage{amsthm}\usepackage{mathrsfs}
\usepackage{comment}
\usepackage[normalem]{ulem}
\usepackage{enumerate}

\usepackage{geometry}
\usepackage{mathtools}

\usepackage{hyperref}
\hypersetup{colorlinks=true,
linktoc=all,linkcolor=black,
citecolor=black}
\usepackage{tikz}
\usepackage{bbm}
\usepackage{mathptmx}
\usepackage{dsfont}
\usepackage{enumitem}

\newcommand{\bea}{\begin{eqnarray}}
\newcommand{\eea}{\end{eqnarray}}

\def\textmatrix#1&#2\\#3&#4\\{\bigl({#1 \atop #3}\ {#2 \atop #4}\bigr)}
\def\dispmatrix#1&#2\\#3&#4\\{\left({#1 \atop #3}\ {#2 \atop #4}\right)}
\newcommand{\be}{\begin{equation}}
\newcommand{\ee}{\end{equation}}
\newcommand{\ben}{\begin{eqnarray*}}
\newcommand{\een}{\end{eqnarray*}}

\newcommand{\bi}{\begin{itemize}}
\newcommand{\ei}{\end{itemize}}

\numberwithin{equation}{section}

\newtheorem{theorem}{Theorem}[section]

\newtheorem{corollary}[theorem]{Corollary}

\newtheorem{definition}[theorem]{Definition}

\newtheorem{proposition}{Proposition}[section]

\theoremstyle{remark}
\newtheorem{remark}[theorem]{Remark}
\theoremstyle{plain}

\newtheorem{lemma}[theorem]{Lemma}
\DeclareMathOperator{\GFF}{GFF_0}

\DeclareMathOperator{\supp}{supp}
\DeclareMathOperator{\spn}{span}
\DeclareMathOperator{\cov}{cov}

\begin{document}
\title[DGFF via Hadamard]%
{Discrete Gaussian Free Field via Hadamard's formula}

\author[H.~Hedenmalm]{Haakan Hedenmalm}
\address[H.~Hedenmalm]{Department of Mathematics and 
Computer Science, St. Petersburg University,
Russia, 14 line of the VO, 29B, 199178 St. Petersburg, Russia;
Beijing Institute of Mathematical Sciences and Applications, 
101408,  Hefangkou Village, Huaibei Town, Huairou District, 544, 
Beijing, China;
Department of Mathematics and Statistics, University of Reading, 
U. K. }
\email{haakan00@gmail.com}
\author[P.~Mozolyako]{Pavel Mozolyako}

\address[P.~Mozolyako]{Department of Mathematics and 
Computer Science, St. Petersburg University,
Russia, 14 line of the VO, 29B, 199178 St. Petersburg, Russia;
Beijing Institute of Mathematical Sciences and Applications, 
101408,  Hefangkou Village, Huaibei Town, Huairou District, 544, 
Beijing, China.
}
\email{pmzlcroak@gmail.com}

\author[D.~Panov]{Daniil Panov}

\address[D.~Panov]{Department of Mathematics and Computer 
Science, St. Petersburg University,
Russia, 14 line of the VO, 29B, 199178 St. Petersburg, Russia}
\email{panovdan2003@gmail.com}

\thanks{The work is supported by Ministry of
Science and Higher Education of the Russian Federation under agreement 
No. 075-15-2025-013}

\begin{abstract}
We establish a discrete Hadamard variational formula for Green
operators on weighted graphs along an increasing family of finite
subgraphs. The increment of the Green operator factors through
harmonic extension from the newly added layer and the positive
square root of its boundary Green operator. This gives a
dimension-free decomposition of the Green kernel and an associated
isometric Hadamard operator. As an application, the discrete Gaussian
free field is constructed by successively adding independent
harmonically extended Gaussian contributions. The formulation uses
only the graph structure and its conductances, without smoothness
assumptions. For nested square-lattice approximations of smooth
planar domains, we also prove convergence of the growing fields in
finite-dimensional distributions and convergence of the integrated
Hadamard increments as Green-operator forms.
\end{abstract}

\maketitle
\section{Introduction}\label{S:Intro}
\subsection{The discrete Hadamard principle}
The classical Hadamard variational formula describes the change
of a Green function under a variation of its domain. Its usual
formulation involves boundary geometry and normal derivatives.
Here we seek an analogue which is intrinsic to a weighted graph
and an increasing sequence of finite subgraphs. In this setting
the variation is a discrete increment rather than an infinitesimal
derivative.

The central result is Theorem~\ref{Hadamard}. With the normalized
Green operators $\tilde{\mathbf G}_n$ associated with the
growing clusters, it yields the positive factorization
\[
\tilde{\mathbf G}_n-\tilde{\mathbf G}_{n-1}
=\mathbf K_n\mathbf K_n^*,
\]
where $\mathbf K_n$ is obtained by composing harmonic extension
from the new layer with the positive square root of the boundary
Green operator. The identity is understood on the common ambient
vertex space by extending the Green operators by zero. Iteration
gives a decomposition of the Green kernel into contributions from
successive layers.

The construction requires neither an ambient Euclidean space nor
a smooth boundary. Its geometric data are encoded in the adjacency
relation and the conductances. In this way, the discrete
Green-operator increments provide a common framework for
Hadamard variation and its Gaussian-field realization.

The growth of the graph is illustrated in Figure~\ref{fig:discrete-foliation}. The black vertices form the preceding cluster and the blue vertices constitute the newly added layer.

\begin{figure}[!htbp]
\centering
\begin{tikzpicture}[x=0.82cm,y=0.82cm, line cap=round, line join=round]
\definecolor{newlayer}{RGB}{27,91,171}
\definecolor{outergray}{RGB}{190,194,200}
\draw[black!65,line width=0.48pt] (0.589,0.000)--(-0.314,0.353);
\draw[black!65,line width=0.48pt] (0.589,0.000)--(-0.325,-0.365);
\draw[black!65,line width=0.48pt] (0.589,0.000)--(1.163,0.135);
\draw[black!65,line width=0.48pt] (0.589,0.000)--(0.448,0.807);
\draw[black!65,line width=0.48pt] (0.589,0.000)--(0.712,-0.546);
\draw[black!65,line width=0.48pt] (-0.314,0.353)--(-0.325,-0.365);
\draw[black!65,line width=0.48pt] (-0.314,0.353)--(0.448,0.807);
\draw[black!65,line width=0.48pt] (-0.314,0.353)--(-0.843,0.646);
\draw[black!65,line width=0.48pt] (-0.325,-0.365)--(-1.272,-0.148);
\draw[black!65,line width=0.48pt] (-0.325,-0.365)--(-0.382,-0.687);
\draw[black!65,line width=0.48pt] (-0.325,-0.365)--(0.712,-0.546);
\draw[black!65,line width=0.48pt] (1.163,0.135)--(0.448,0.807);
\draw[black!65,line width=0.48pt] (1.163,0.135)--(0.712,-0.546);
\draw[black!65,line width=0.48pt] (1.163,0.135)--(1.985,0.116);
\draw[black!65,line width=0.48pt] (1.163,0.135)--(1.575,-0.609);
\draw[black!65,line width=0.48pt] (0.448,0.807)--(-0.843,0.646);
\draw[black!65,line width=0.48pt] (0.448,0.807)--(1.634,0.925);
\draw[black!65,line width=0.48pt] (0.448,0.807)--(0.429,1.217);
\draw[black!65,line width=0.48pt] (-0.843,0.646)--(-1.272,-0.148);
\draw[black!65,line width=0.48pt] (-0.843,0.646)--(1.634,0.925);
\draw[black!65,line width=0.48pt] (-0.843,0.646)--(-0.806,1.221);
\draw[black!65,line width=0.48pt] (-0.843,0.646)--(-1.683,0.651);
\draw[black!65,line width=0.48pt] (-1.272,-0.148)--(-0.382,-0.687);
\draw[black!65,line width=0.48pt] (-1.272,-0.148)--(-2.211,-0.129);
\draw[black!65,line width=0.48pt] (-0.382,-0.687)--(0.712,-0.546);
\draw[black!65,line width=0.48pt] (-0.382,-0.687)--(-2.211,-0.129);
\draw[black!65,line width=0.48pt] (-0.382,-0.687)--(-1.708,-0.967);
\draw[black!65,line width=0.48pt] (-0.382,-0.687)--(-0.445,-1.263);
\draw[black!65,line width=0.48pt] (0.712,-0.546)--(0.902,-1.366);
\draw[black!65,line width=0.48pt] (0.712,-0.546)--(1.575,-0.609);
\draw[black!65,line width=0.48pt] (1.985,0.116)--(1.634,0.925);
\draw[black!65,line width=0.48pt] (1.985,0.116)--(1.575,-0.609);
\draw[newlayer!75,line width=0.48pt] (1.985,0.116)--(3.085,0.164);
\draw[black!65,line width=0.48pt] (1.634,0.925)--(0.429,1.217);
\draw[newlayer!75,line width=0.48pt] (1.634,0.925)--(2.650,1.089);
\draw[black!65,line width=0.48pt] (0.429,1.217)--(-0.806,1.221);
\draw[newlayer!75,line width=0.48pt] (0.429,1.217)--(2.650,1.089);
\draw[newlayer!75,line width=0.48pt] (0.429,1.217)--(1.315,1.481);
\draw[newlayer!75,line width=0.48pt] (0.429,1.217)--(0.117,1.955);
\draw[black!65,line width=0.48pt] (-0.806,1.221)--(-1.683,0.651);
\draw[newlayer!75,line width=0.48pt] (-0.806,1.221)--(-1.256,1.707);
\draw[black!65,line width=0.48pt] (-1.683,0.651)--(-2.211,-0.129);
\draw[newlayer!75,line width=0.48pt] (-1.683,0.651)--(-2.556,1.240);
\draw[black!65,line width=0.48pt] (-2.211,-0.129)--(-1.708,-0.967);
\draw[newlayer!75,line width=0.48pt] (-2.211,-0.129)--(-2.556,1.240);
\draw[newlayer!75,line width=0.48pt] (-2.211,-0.129)--(-2.645,0.275);
\draw[newlayer!75,line width=0.48pt] (-2.211,-0.129)--(-2.756,-0.599);
\draw[black!65,line width=0.48pt] (-1.708,-0.967)--(-0.445,-1.263);
\draw[newlayer!75,line width=0.48pt] (-1.708,-0.967)--(-2.174,-1.470);
\draw[black!65,line width=0.48pt] (-0.445,-1.263)--(0.902,-1.366);
\draw[newlayer!75,line width=0.48pt] (-0.445,-1.263)--(-0.823,-1.851);
\draw[black!65,line width=0.48pt] (0.902,-1.366)--(1.575,-0.609);
\draw[newlayer!75,line width=0.48pt] (0.902,-1.366)--(-0.823,-1.851);
\draw[newlayer!75,line width=0.48pt] (0.902,-1.366)--(0.554,-1.742);
\draw[newlayer!75,line width=0.48pt] (0.902,-1.366)--(1.641,-1.298);
\draw[newlayer!75,line width=0.48pt] (1.575,-0.609)--(2.930,-0.805);
\draw[newlayer!75,line width=0.48pt] (3.085,0.164)--(2.650,1.089);
\draw[newlayer!75,line width=0.48pt] (3.085,0.164)--(2.930,-0.805);
\draw[outergray!70,line width=0.48pt] (3.085,0.164)--(3.674,0.000);
\draw[newlayer!75,line width=0.48pt] (2.650,1.089)--(1.315,1.481);
\draw[outergray!70,line width=0.48pt] (2.650,1.089)--(3.858,0.968);
\draw[newlayer!75,line width=0.48pt] (1.315,1.481)--(0.117,1.955);
\draw[outergray!70,line width=0.48pt] (1.315,1.481)--(3.858,0.968);
\draw[outergray!70,line width=0.48pt] (1.315,1.481)--(2.497,1.474);
\draw[outergray!70,line width=0.48pt] (1.315,1.481)--(1.771,2.303);
\draw[newlayer!75,line width=0.48pt] (0.117,1.955)--(-1.256,1.707);
\draw[outergray!70,line width=0.48pt] (0.117,1.955)--(0.367,2.563);
\draw[newlayer!75,line width=0.48pt] (-1.256,1.707)--(-2.556,1.240);
\draw[outergray!70,line width=0.48pt] (-1.256,1.707)--(-1.114,2.534);
\draw[newlayer!75,line width=0.48pt] (-2.556,1.240)--(-2.645,0.275);
\draw[outergray!70,line width=0.48pt] (-2.556,1.240)--(-1.114,2.534);
\draw[outergray!70,line width=0.48pt] (-2.556,1.240)--(-1.970,1.689);
\draw[outergray!70,line width=0.48pt] (-2.556,1.240)--(-3.047,1.221);
\draw[newlayer!75,line width=0.48pt] (-2.645,0.275)--(-2.756,-0.599);
\draw[outergray!70,line width=0.48pt] (-2.645,0.275)--(-3.243,0.393);
\draw[newlayer!75,line width=0.48pt] (-2.756,-0.599)--(-2.174,-1.470);
\draw[outergray!70,line width=0.48pt] (-2.756,-0.599)--(-3.853,-0.466);
\draw[newlayer!75,line width=0.48pt] (-2.174,-1.470)--(-0.823,-1.851);
\draw[outergray!70,line width=0.48pt] (-2.174,-1.470)--(-3.853,-0.466);
\draw[outergray!70,line width=0.48pt] (-2.174,-1.470)--(-3.033,-1.216);
\draw[outergray!70,line width=0.48pt] (-2.174,-1.470)--(-2.473,-2.121);
\draw[newlayer!75,line width=0.48pt] (-0.823,-1.851)--(0.554,-1.742);
\draw[outergray!70,line width=0.48pt] (-0.823,-1.851)--(-0.981,-2.233);
\draw[newlayer!75,line width=0.48pt] (0.554,-1.742)--(1.641,-1.298);
\draw[outergray!70,line width=0.48pt] (0.554,-1.742)--(0.364,-2.540);
\draw[newlayer!75,line width=0.48pt] (1.641,-1.298)--(2.930,-0.805);
\draw[outergray!70,line width=0.48pt] (1.641,-1.298)--(0.364,-2.540);
\draw[outergray!70,line width=0.48pt] (1.641,-1.298)--(1.793,-2.332);
\draw[outergray!70,line width=0.48pt] (1.641,-1.298)--(2.561,-1.512);
\draw[outergray!70,line width=0.48pt] (2.930,-0.805)--(3.494,-0.876);
\draw[outergray!70,line width=0.48pt] (3.674,0.000)--(3.494,-0.876);
\draw[outergray!70,line width=0.48pt] (3.858,0.968)--(2.497,1.474);
\draw[outergray!70,line width=0.48pt] (2.497,1.474)--(1.771,2.303);
\draw[outergray!70,line width=0.48pt] (0.367,2.563)--(-1.114,2.534);
\draw[outergray!70,line width=0.48pt] (-1.114,2.534)--(-1.970,1.689);
\draw[outergray!70,line width=0.48pt] (-3.047,1.221)--(-3.243,0.393);
\draw[outergray!70,line width=0.48pt] (-3.243,0.393)--(-3.853,-0.466);
\draw[outergray!70,line width=0.48pt] (-3.033,-1.216)--(-2.473,-2.121);
\draw[outergray!70,line width=0.48pt] (-2.473,-2.121)--(-0.981,-2.233);
\draw[outergray!70,line width=0.48pt] (0.364,-2.540)--(1.793,-2.332);
\draw[outergray!70,line width=0.48pt] (1.793,-2.332)--(2.561,-1.512);
\filldraw[fill=black,draw=black,line width=.35pt] (0.589,0.000) circle (1.9pt);
\filldraw[fill=black,draw=black,line width=.35pt] (-0.314,0.353) circle (1.9pt);
\filldraw[fill=black,draw=black,line width=.35pt] (-0.325,-0.365) circle (1.9pt);
\filldraw[fill=black,draw=black,line width=.35pt] (1.163,0.135) circle (1.9pt);
\filldraw[fill=black,draw=black,line width=.35pt] (0.448,0.807) circle (1.9pt);
\filldraw[fill=black,draw=black,line width=.35pt] (-0.843,0.646) circle (1.9pt);
\filldraw[fill=black,draw=black,line width=.35pt] (-1.272,-0.148) circle (1.9pt);
\filldraw[fill=black,draw=black,line width=.35pt] (-0.382,-0.687) circle (1.9pt);
\filldraw[fill=black,draw=black,line width=.35pt] (0.712,-0.546) circle (1.9pt);
\filldraw[fill=black,draw=black,line width=.35pt] (1.985,0.116) circle (1.9pt);
\filldraw[fill=black,draw=black,line width=.35pt] (1.634,0.925) circle (1.9pt);
\filldraw[fill=black,draw=black,line width=.35pt] (0.429,1.217) circle (1.9pt);
\filldraw[fill=black,draw=black,line width=.35pt] (-0.806,1.221) circle (1.9pt);
\filldraw[fill=black,draw=black,line width=.35pt] (-1.683,0.651) circle (1.9pt);
\filldraw[fill=black,draw=black,line width=.35pt] (-2.211,-0.129) circle (1.9pt);
\filldraw[fill=black,draw=black,line width=.35pt] (-1.708,-0.967) circle (1.9pt);
\filldraw[fill=black,draw=black,line width=.35pt] (-0.445,-1.263) circle (1.9pt);
\filldraw[fill=black,draw=black,line width=.35pt] (0.902,-1.366) circle (1.9pt);
\filldraw[fill=black,draw=black,line width=.35pt] (1.575,-0.609) circle (1.9pt);
\filldraw[fill=newlayer,draw=newlayer,line width=.35pt] (3.085,0.164) circle (1.9pt);
\filldraw[fill=newlayer,draw=newlayer,line width=.35pt] (2.650,1.089) circle (1.9pt);
\filldraw[fill=newlayer,draw=newlayer,line width=.35pt] (1.315,1.481) circle (1.9pt);
\filldraw[fill=newlayer,draw=newlayer,line width=.35pt] (0.117,1.955) circle (1.9pt);
\filldraw[fill=newlayer,draw=newlayer,line width=.35pt] (-1.256,1.707) circle (1.9pt);
\filldraw[fill=newlayer,draw=newlayer,line width=.35pt] (-2.556,1.240) circle (1.9pt);
\filldraw[fill=newlayer,draw=newlayer,line width=.35pt] (-2.645,0.275) circle (1.9pt);
\filldraw[fill=newlayer,draw=newlayer,line width=.35pt] (-2.756,-0.599) circle (1.9pt);
\filldraw[fill=newlayer,draw=newlayer,line width=.35pt] (-2.174,-1.470) circle (1.9pt);
\filldraw[fill=newlayer,draw=newlayer,line width=.35pt] (-0.823,-1.851) circle (1.9pt);
\filldraw[fill=newlayer,draw=newlayer,line width=.35pt] (0.554,-1.742) circle (1.9pt);
\filldraw[fill=newlayer,draw=newlayer,line width=.35pt] (1.641,-1.298) circle (1.9pt);
\filldraw[fill=newlayer,draw=newlayer,line width=.35pt] (2.930,-0.805) circle (1.9pt);
\filldraw[fill=outergray!45,draw=outergray,line width=.35pt] (3.674,0.000) circle (1.9pt);
\filldraw[fill=outergray!45,draw=outergray,line width=.35pt] (3.858,0.968) circle (1.9pt);
\filldraw[fill=outergray!45,draw=outergray,line width=.35pt] (2.497,1.474) circle (1.9pt);
\filldraw[fill=outergray!45,draw=outergray,line width=.35pt] (1.771,2.303) circle (1.9pt);
\filldraw[fill=outergray!45,draw=outergray,line width=.35pt] (0.367,2.563) circle (1.9pt);
\filldraw[fill=outergray!45,draw=outergray,line width=.35pt] (-1.114,2.534) circle (1.9pt);
\filldraw[fill=outergray!45,draw=outergray,line width=.35pt] (-1.970,1.689) circle (1.9pt);
\filldraw[fill=outergray!45,draw=outergray,line width=.35pt] (-3.047,1.221) circle (1.9pt);
\filldraw[fill=outergray!45,draw=outergray,line width=.35pt] (-3.243,0.393) circle (1.9pt);
\filldraw[fill=outergray!45,draw=outergray,line width=.35pt] (-3.853,-0.466) circle (1.9pt);
\filldraw[fill=outergray!45,draw=outergray,line width=.35pt] (-3.033,-1.216) circle (1.9pt);
\filldraw[fill=outergray!45,draw=outergray,line width=.35pt] (-2.473,-2.121) circle (1.9pt);
\filldraw[fill=outergray!45,draw=outergray,line width=.35pt] (-0.981,-2.233) circle (1.9pt);
\filldraw[fill=outergray!45,draw=outergray,line width=.35pt] (0.364,-2.540) circle (1.9pt);
\filldraw[fill=outergray!45,draw=outergray,line width=.35pt] (1.793,-2.332) circle (1.9pt);
\filldraw[fill=outergray!45,draw=outergray,line width=.35pt] (2.561,-1.512) circle (1.9pt);
\filldraw[fill=outergray!45,draw=outergray,line width=.35pt] (3.494,-0.876) circle (1.9pt);
\node[font=\small,fill=white,inner sep=1.5pt] at (-1.05,0.18) {$\Gamma_{n-1}$};
\node[font=\small,text=newlayer,fill=white,inner sep=1.5pt] at (2.6,1.0) {$\gamma_n$};
\end{tikzpicture}
\caption{Growth along a discrete foliation of an irregular graph. Black vertices belong to the preceding cluster $\Gamma_{n-1}$, blue vertices form the added layer $\gamma_n$, and gray vertices lie outside $\Gamma_n$. Edges join vertices in the same or adjacent layers.}
\label{fig:discrete-foliation}
\end{figure}

\subsection{The DGFF as an application}
The Discrete Gaussian Free Field (DGFF) on a finite subgraph
can be realized as a random linear combination of an orthonormal
basis of its discrete Dirichlet space, with independent standard
Gaussian coefficients. Its covariance is the normalized Green
kernel. The Hadamard factorization provides a compatible choice
of bases as the graph grows. In particular, the associated
Hadamard operator is an isometric isomorphism
(Theorem 4.3), and Theorem~\ref{t:main51}
constructs the DGFF by adding independent, harmonically
extended Gaussian contributions from successive layers.

This parallels the continuum construction of Hedenmalm and
Nieminen~\cite{HN}, which interprets the classical Hadamard
formula in terms of Gaussian free field growth. The present
formulation makes the layerwise Green-operator factorization
available on general weighted graphs.

\subsection{Related Gaussian decompositions}
Related decompositions of the Gaussian free field appear in
Biskup, Ding, and Goswami~\cite[Section~3.3]{BDG} and in
Duminil-Copin, Goswami, Rodriguez, and
Severo~\cite[Section~3.1]{DGRS}. These constructions arise in
specific probabilistic applications. Our contribution is an intrinsic, positive Hadamard
factorization of Green-operator increments on general
weighted graphs, together with its associated isometric
operator and growing-field construction.

\subsection{Relation to continuum limits}
The positive Green-operator increments provide a natural
framework for studying Hadamard variation under approximation
of Dirichlet spaces. Section~\ref{S:SCDGFF} proves a square-lattice scaling limit
for smooth planar domains. It identifies the limiting growing GFF
in finite-dimensional distributions and shows that integrated
discrete Hadamard increments converge as Green-operator forms
(Corollary~\ref{cor:integrated-hadamard}). 

\subsection{Organization of the paper}
Section~\ref{S:DSN} introduces weighted graphs, Dirichlet
spaces, Green kernels, and discrete foliations.
Section~\ref{S:ONB} develops the Hilbert-space viewpoint
behind the growing construction. Section~\ref{S:DHF}
establishes the discrete Hadamard formula and the associated
isometric operator. Section~\ref{S:GDGFF} constructs the DGFF
from independent layerwise contributions, and
Section~\ref{S:PDGFF} records further properties of the
resulting process. Finally, Section~\ref{S:SCDGFF}
establishes the square-lattice scaling limit and convergence of
integrated Hadamard increments.

\section{Discrete setting and notation}\label{S:DSN}

\subsection{Purpose of this section}
Here, we introduce the discrete setting we are going to work with and the related notation.

\subsection{Graphs with conductances}
Let $\Gamma=(V,E)$ be a connected locally finite graph. Here $V$ stands for the collection 
of vertices, and $E$ for the collection of edges. We will frequently write $V = V(\Gamma)$ and 
$E = E(\Gamma)$ to indicate that $V$ and $E$ are associated with the given graph $\Gamma$. 
Edges connect a pair of vertices, and we write $e = (x,y)$ for the edge connecting vertex $x$ 
with the different vertex $y$. For instance, we assume that for any $x\in V$: $(x,x)\notin E$ 
and for any edge $e=(x,y)$ there is a reverse edge $(y,x)\in E$, denoted by $(-e)$. In other 
words, the set $E$ of edges is symmetric and irreflexive. We will write $x\sim y$ to express 
that $(x,y)$ is in the edge set $E$, and think of $x$ and $y$ as \textit{neighbors}. If there are
competing graphs to be concerned with, we use a subscript to indicate with respect to which
graph the given relation is taken, e.g. $x\sim_\Gamma y$. 
Given an edge $e\in E$, we express the vertices that it connects by $e^{-}$ and $e^{+}$, 
in the sense that if $e = (x,y)$ then $x = e^{-}$ and $y = e^{+}$. Let $c$ be a \textit{conductance} 
\cite[Chapter 2]{PL}, that is, a positive symmetric weight on the edges of $\Gamma$, and 
for $x\in V$, let $ \pi(x)=\sum_{y:y\sim x}c(x,y)$ be the stationary distribution for the 
corresponding random walk.

\subsection{Hilbert spaces on vertices and on edges}
For a finite subset $U\subset V$ of vertices, we define 
the real Hilbert space of functions on vertices 
\begin{equation}\notag
\ell^2(U):= \{f:V\to\mathbb{R},\ \supp f\subset U \}
\end{equation}
with the inner product  
\begin{equation}
\langle f,g\rangle_{U}:=\sum_{x\in U}f(x)g(x).
\end{equation}  
Elements $f\in\ell^2(U)$ are tacitly assumed extended to all vertices in $V$ by declaring that
$f(x)=0$ for all $x\in V\setminus U$.
We shall need also the corresponding concept of Hilbert spaces on the edges.  So, for a 
finite symmetric set $F \subset E$ of edges (i.e., $e\in F \text{ if and only if } (-e)\in F$), we 
define the real Hilbert space of antisymmetric functions on those edges  
\begin{equation}\notag
\ell^2_{-}(F):=\big\{\phi: \to \mathbb{R}:\; \forall e\in F:\; \phi(e)=-\phi(-e)\big\},
\end{equation} 
supplied with the inner product
\begin{equation}
\displaystyle\langle\phi,\psi\rangle_{F^{-}}:=\frac{1}{2}\sum_{e\in F}\phi(e)\psi(e).
\end{equation}
As with the vertices, we let the elements $\phi\in\ell^2_{-}(F)$ extend to all the edges in $E$
by declaring that $\phi(e)=0$ if $e\in E\setminus F$.

\subsection{Weighted boundary and coboundary operators}

We now define the (weighted coboundary) operator $\mathbf{d}$:
\begin{equation}
 \mathbf{d}f(x,y)=\sqrt{c(x,y)}(f(x)-f(y)).
\end{equation} 
As defined, the coboundary operator $\mathbf{d}$ takes functions defined on the vertices 
and produces antisymmetric functions defined on the edges.
We also define the associated (weighted boundary) operator $\mathbf{d}^*$:
\begin{equation}
\displaystyle \mathbf{d}^*\phi(x) =\sum_{e\in E:\, e^-=x} \sqrt{c(e)}\phi(e).
\end{equation}
For a finite subset $U \subset V$, we denote by $E_U$ the collection of edges in the 
given graph $\Gamma$ that have at least one endpoint in $U$. 
We then immediately obtain the following statement.

\begin{proposition}
\label{adjoints}
Let $f\in \ell^2(U),\phi\in \ell^2_{-}(E_U)$. Then:
\begin{equation}
\langle f,\mathbf{d}^{*}\phi\rangle_{U}= \langle \mathbf{d}f,\phi\rangle_{E_U^-}.
\end{equation} 
Moreover, if $g$ is any function on $V$, then 
\[
\mathbf{d}^*\mathbf{d}\,g(x)=\sum_{y:y\sim x}c(x,y)(g(x)-g(y)).
\]
\end{proposition}

Next, we fix an arbitrary finite nonempty subset $U \subset V$. We define the 
\emph{discrete Sobolev space} 
\begin{equation}\notag
H_0^1(U)=\{f:V\to\mathbb{R}:\,\, \supp f \subset U \},
\end{equation}
where the \emph{support} has the standard discrete meaning:
\[
\supp f:=\{x\in V:\,\,f(x)\ne0\}.
\]
We endow $H^1_0(U)$ with the inner product structure given by
\begin{equation}
\langle f,g\rangle_{\nabla, U}:=\langle \mathbf{d}f,\mathbf{d}g\rangle_{E_U^-}
=\langle \mathbf{d}^*\mathbf{d}f,g\rangle_{U}.
\end{equation}
The associated Hilbert space norm
\begin{equation}
\| f\|^2_{\nabla, U}:=\langle \mathbf{d}f,\mathbf{d}f\rangle_{\nabla,U}
=\langle \mathbf{d}f,\mathbf{d}f\rangle_{E_U^-}
\end{equation}
then extends naturally to a seminorm on all functions $f:V\to\mathbb{R}$.
It is often referred to as the \emph{discrete Dirichlet energy} on $U$.

\subsection{The discrete Laplacian}
It is natural to introduce the operator 
$\Delta_U:\ell^2(U)\to \ell^2(U)$:
\begin{equation}
\Delta_Uf(x):=\mathds{1}_{U}\,
\mathbf{d}^*\mathbf{d}f(x),
\end{equation}
where $\mathds{1}_{U}$ denotes the characteristic function of the subset 
$U\subset V$. As a matter of definition, we say that a function $f:V\to\mathbb{R}$
is \emph{discrete harmonic on $U$} if $\Delta_U f=0$ holds.
Such functions are critical points of the discrete Dirichlet energy on $U$ under 
additive perturbations using functions from $H^1_0(U)$.

According to the properties $(a)$ and $(b)$ of the following
proposition, the operator $\Delta_U$ defines a self-adjoint operator
with nonnegative eigenvalues on the finite-dimensional space $\ell^2(U)$,
in which case it has a naturally defines square root 
$\displaystyle (\Delta_U)^{\frac{1}{2}}$ 
with nonnegative eigenvalues. Hence part $(c)$ of the same proposition also 
makes sense.

\begin{proposition}\label{ModLapl}
The following properties hold for each finite subset $U\subset V$:
\begin{enumerate}[label=(\alph*)]
\item[$(a)$]  $\langle\Delta_U f,g\rangle_{U}=\langle f,
\Delta_Ug\rangle_{U}\;$  for any $f,g \in \ell^2(U)$.
\item[$(b)$] All eigenvalues of the operator $\Delta_U$ are 
nonnegative.
\item[$(c)$] The operator $\displaystyle (\Delta_U)^{\frac{1}{2}}$ 
defines an isometry $H_0^1(U)\to \ell^2(U)$.
\end{enumerate}
\end{proposition}

\proof[Proof]
We note that
\begin{equation}
\langle\Delta_Uf,g\rangle_{U}=\langle \mathbf{d}^*\mathbf{d}f,g\rangle_{U}=
\langle \mathbf{d}f,\mathbf{d}g\rangle_{E_U^-}= 
\langle f,
\mathbf{d}^*\mathbf{d}\,g\rangle_{U}=\langle f,\Delta_Ug\rangle_{U},
\end{equation}
which shows that the assertion $(a)$ holds.
Next, let $\lambda$ be an eigenvalue of the self-adjoint operator 
$\Delta_U$ and let $f\in \ell^2(U)$ be a corresponding eigenfunction. Since
\begin{equation}
\lambda\langle f,f\rangle_{U}=\langle\Delta_Uf,f\rangle_{U}= 
\langle \mathbf{d}^*\mathbf{d}f,f\rangle_{U}
=\langle \mathbf{d}f,\mathbf{d}f\rangle_{E_U^-}=\langle f,f\rangle_{\nabla, U}\geq 0,
\end{equation}
it follows that $\lambda \geq 0$, and assertion $(b)$ follows as well.
Finally, given any two functions $f,g \in H^1_0(U)$ we obtain that
\begin{equation}
\big\langle(\Delta_U)^{\frac{1}{2}}f,
(\Delta_U)^{\frac{1}{2}}g\big\rangle_{U}=\langle\Delta_Uf,g\rangle_{U}
=\langle \mathbf{d}^*\mathbf{d}\,f,g\rangle_{U}=\langle f,g\rangle_{H_0^1(U)}.
\end{equation}
It is now immediate that the operator $(\Delta_U)^{\frac{1}{2}}$ constitutes an 
isometry $H_0^1(U) \to \ell^2(U)$, and property $(c)$ follows as well.
\endproof

\subsection{The discrete Green function}
For $y\in U$, we define the \emph{discrete Green function} $G_U(\cdot,y)$ as 
the solution to the following system: 
\begin{equation}
\begin{cases}
\mathbf{d}^*\mathbf{d}\, G_U(\cdot,y)=\pi(y)\,\delta_{y}\,\,\,\text{on}\,\,\, U,\\
G_U(x,y)=0, \qquad x \in V \setminus U,
\end{cases}
\end{equation}
where we recall that $\pi(x)=\sum_{y:y\sim x}c(y,x)$ stands for the stationary 
distribution for the random walk. Here, of course, $\delta_y(x)$ is understood in the
Kronecker sense, so that $\delta_y(x)=0$ of $x\neq y$ while 
$\delta_y(y)=1$.
We also declare that $G_U(x,y)=0$ holds for $y\in V \setminus U,\; x\in V$. 
We refer to \cite[Chapter 2]{PL} on the existence and uniqueness of such a function. 
The normalized Green operator $\tilde{\mathbf{G}}_U:\ell^2(U) \to \ell^2(U)$ is
given by
\begin{equation}
\tilde{\mathbf{G}}_Uf(x)=\sum_{y\in U}\tilde{G}_U(x,y)f(y),
\end{equation}
where
\begin{equation}
\tilde{G}_U(x,y) = \frac{G_U(x,y)}{\pi(y)}.
\end{equation}
Then, by definition, $\Delta_U\tilde{\mathbf{G}}_U=\mathbf{id}_{\ell^2(U)}$ is the 
identity operator on $\ell^2(U)$, so that $\tilde{\mathbf{G}}_U$ constitutes a right inverse 
to $\Delta_U$. Taking adjoints, we obtain
\begin{equation}
\tilde{\mathbf{G}}_U^\ast\Delta_U
=\tilde{\mathbf{G}}_U^\ast\Delta_U^\ast=
\mathbf{id}_{\ell^2(U)},
\label{eq:adjoint-1}
\end{equation}
since $\Delta_U^\ast=\Delta_U$ holds in accordance with Proposition 
\ref{ModLapl}. Then, multiplying by $\tilde{\mathbf{G}}_U$ from the right, we find that
\[
\tilde{\mathbf{G}}_U^\ast=\tilde{\mathbf{G}}_U^\ast\Delta_U\tilde{\mathbf{G}}_U
=\tilde{\mathbf{G}}_U^\ast\Delta_U^\ast\tilde{\mathbf{G}}_U=
\tilde{\mathbf{G}}_U,
\]
making $\tilde{\mathbf{G}}_U$ self-adjoint as an operator on $\ell^2(U)$. 
This property amounts to the following \emph{symmetry property} of the Green function 
$G_U$:
\begin{equation}
\pi(x)G_U(x,y)=\pi(y)G_U(y,x),\qquad x,y\in V.
\label{eq:G-symm-1}
\end{equation}
We also see from \eqref{eq:adjoint-1} that $\tilde{\mathbf{G}}_U$ is also a left inverse
to $\Delta_U$ on $\ell^2(U)$, so that 
$\tilde{\mathbf{G}}_U=\Delta_U^{-1}$ as operators on $\ell^2(U)$. 

\subsection{The discrete Poisson kernel}
In addition to the discrete Green function $G_U$, we shall need the associated discrete
Poisson kernel as well. This is perhaps less canonical than the concept of the Green 
function, but we may proceed as follows. 
For a nonempty subset $W \subset U$, we define the \emph{discrete Poisson kernel} 
$P_{U,W}(\cdot,y)$, for a given point $y\in W$, as the solution to the following system: 
\begin{equation}
\begin{cases}
\mathbf{d}^*\mathbf{d}\, P_{U,W}(\cdot,y)=0\,\,\,\,\,\text{on}\,\,\,\,\, U\setminus W,
\\
P_{U,W}(x,y)=\delta_y(x), \qquad x \in W.
\end{cases}
\end{equation}
Moreover, for $x \in V\setminus U$, we declare that $P_{U,W}(x,y)=0$. 
Associated with the discrete Poisson kernel, we have the \emph{Poisson extension
operator} $\mathbf{P}_{U,W}:\ell^2(W)\to\ell^2(U)\subset\ell^2(V)$ given by the formula
\[
\mathbf{P}_{U,W}f(x)=\sum_{y\in W}P_{U,W}(x,y)\,f(y),\qquad x\in V.
\]
While $\mathbf{P}_{U,W}f$ automatically vanishes on $V\setminus U$, we then have
that $\mathbf{P}_{U,W}f=f$ on $W$ while $\mathbf{P}_{U,W}f$ is harmonic on 
$U\setminus W$. As such, it is unique. 
Regarding existence and uniqueness issues, we again refer to \cite[Chapter 2]{PL}.

\subsection{Foliating the graph $\Gamma$}
We turn to the context we are going to work with to develop the Hadamard formula.
First we need the notion of a \emph{subgraph}.

\begin{definition}
\rm We say that $\gamma=(V(\gamma),E(\gamma))$ is a \emph{subgraph} of 
$\Gamma=(V,E)$  if  $V(\gamma)\subset V$ and $E(\gamma)\subset E$, while 
if $(x,y)\in E(\gamma)$, then necessarily $x,y\in V(\gamma)$. Moreover, we say
that $\gamma$ is an \emph{exact subgraph} of $\Gamma$ if it is a subgraph,
and whenever $x,y\in V(\gamma)$ and $(x,y)\in E$, then $(x,y)\in E(\gamma)$. 
\end{definition}

\begin{definition}
\rm
A sequence $\{\gamma_n\}_{n}$, indexed by $n\in\mathbb{Z}_{\ge0}$ and 
consisting of non-empty finite subgraphs of $\Gamma$, is called a 
\textit{discrete foliation}, if it has the following properties:
\begin{enumerate}[leftmargin=0.65cm,label=(\alph*)]
\item[$(a)$] For each $n$, $\gamma_n$ is an exact subgraph.
\item[$(b)$] For $m \neq n$, the associated sets of vertices $V(\gamma_{m})$ and 
$V(\gamma_{n})$ are disjoint.
\item[$(c)$] If $x\in V(\gamma_0)$ and $y\sim_\Gamma x$, then $y\in V(\gamma_0)\cup
V(\gamma_1)$. 
\item[$(d)$] For $n=1,2,3,\ldots$: If $x\in V(\gamma_n)$ and $y\sim_\Gamma x$, 
then $y\in V(\gamma_{n-1})\cup V(\gamma_n)\cup V(\gamma_{n+1})$.
\item[$(e)$] We have that $\displaystyle \bigcup_{n}V(\gamma_n)=V(\Gamma)$ .
\end{enumerate}
\end{definition}

In the context of the above definition, we refer to the individual subgraphs 
$\gamma_n$  as \emph{foliating layers}.

\begin{definition}
\rm
Given a discrete foliation $\{\gamma_n\}_n$ indexed by 
$n\in\mathbb{Z}_{\ge0}$,
we define the associated \emph{foliation growth cluster} $\Gamma_n$ of 
$\Gamma$ as the union of all preceding foliating layers  $\gamma_k$, with 
$k=0,\ldots,n$, of the given discrete foliation, in 
the following sense:  the vertices of $\Gamma_n$ are 
\begin{equation}\notag
V(\Gamma_n)=\bigcup_{0\leq k\leq n}V(\gamma_k),
\end{equation}
while the set of edges $E(\Gamma_n)$ consists of all the edges in $\Gamma$ 
between any two vertices in $V(\Gamma_n)$. 
\end{definition}

\subsection{Simplification of the notational conventions for discrete foliations}
As a matter of convenience, we prefer to write $\ell^2(\Gamma_n)$ and 
$H_0^1(\Gamma_n)$ in place of $\ell^2(V(\Gamma_n))$ and $H_0^1(V(\Gamma_n))$, 
respectively. Analogously, we write $\Delta_n$ for the discrete Laplacian 
$\Delta_{V(\Gamma_n)}$, $G_n$ for the discrete Green function 
$G_{V(\Gamma_n)}$ while $\tilde{G}_n$ stands for the discrete normalized Green 
function $\tilde{G}_{V(\Gamma_n)}$, and $P_n$ denotes the discrete Poisson kernel 
$P_{V(\Gamma_n),V(\gamma_n)}$.

\section{Orthonormal bases in $H^1_0$: construction of the foliation}\label{S:ONB}

\subsection{Aim of the section}
The main purpose of this note is the construction of a certain operator which produces the 
GFF (Gaussian Free Field) out of the harmonic extension of the  WNF (White Noise Field). 
The most straightforward way to do so is to provide an isometry that maps an orthonormal 
basis in $\ell^2$ onto one in $H^1_0$. In \cite{HN} this was carried out via the so-called 
\textit{Hadamard operator}, which, in turn, was extracted from  the
\textit{Hadamard variational formula} itself. In the discrete setting the corresponding 
formula is not readily available, nor can it be immediately obtained by just discretizing 
the corresponding continuous expression, so, in fact, we need to run our argument 
backwards. In other words, we assume that such a discrete isometric operator does exist 
and satisfies some natural properties relevant to our setting, and then we try to deduce 
what it should look like. Along the way, we also reverse-engineer the discrete version 
of the Hadamard variational formula.

\subsection{The GFF and WNF processes}
Before we proceed to the actual construction, let us recall what are the discrete 
WNF (White Noise Field) and $\mathrm{GFF}_0$ (Gaussian Free Field with vanishing 
boundary values) processes. 
Let $\Gamma_U$ be a finite connected subgraph of $\Gamma$, with the vertex set 
$U\subset V$.
\par
We denote the discrete WNF on the subgraph $\Gamma_U$ by 
$\Phi_U$. It is just a random linear combination
\begin{equation}
\Phi_U(x)=\sum_{j}\xi_j\,
\phi_j(x),
\end{equation}
where the set $\{\phi_j\}$ forms an orthonormal basis of $\ell^2(U)$, while $\{\xi_j\}$ are 
independent standard Gaussian random variables (i.e. distributed as $\mathcal{N}(0,1)$). 
Since there are only finitely many vertices in $U$, we see that $\Phi_U$ can be 
thought of as a random Gaussian vector $(\Phi_U(x_1),....,\Phi_U(x_{M}))$, where 
$U=\{x_1,...,x_{M}\}$.
We define the discrete $\mathrm{GFF}_0$ on the subgraph $\Gamma_U$, denoted 
$\Psi_U$, in a similar fashion:
\begin{equation}
\Psi_U(x)=\sum_{j}\xi_j\,
\psi_j(x),\qquad x\in U,
\end{equation}
where the finite sequence $\{\psi_j\}_j$ forms an orthonormal basis of $H_0^1(U)$, and 
$\{\xi_j\}_j$ are independent standard Gaussian random variables. As in the case of the 
discrete WNF, the discrete $\mathrm{GFF}_0$ can also be viewed as a random Gaussian
vector.

The following proposition shows that the distributions of $\Phi_U$ and $\Psi_U$ do not
depend on the choice of the orthonormal bases, since jointly Gaussian vectors are
characterized by first and second moments (including covariances). 
We recall the standard notation
\[
\cov(X,Y)=\mathbb{E}(XY)-\mathbb{E}(X)\mathbb{E}(Y)
\]
for the covariance of two stochastic variables, where $\mathbb{E}$ is the expectation 
operation.

\begin{proposition}
For any $x,y \in U$ we have:
\begin{equation}
\mathbb{E}(\Phi_U(x))=\mathbb{E}(\Psi_U(x))=0,
\tag{i}
\end{equation}
\begin{equation}
\cov(\Phi_U(x),\Phi_U(y))=\delta_y(x),
\tag{ii}\end{equation}
and
\begin{equation}
\cov(\Psi_U(x),\Psi_U(y))=\frac{1}{\pi(y)}G_U(x,y).
\tag{iii}
\end{equation}
\end{proposition}

\begin{proof}
Since the Gaussian variables $\xi_j$ are from $\mathcal{N}(0,1)$, the first moments
vanish, i.e. $(i)$ holds. As to second moments, we find using the independence of 
the Gaussian variables that 
\[
\mathbb{E}\,\Phi_U(x)\Phi_U(y)=\sum_j \phi_j(x)\phi_j(y)=\delta_y(x),
\] 
where the last equality identity uses the orthonormal basis property of $\{\phi_j\}_j$ 
in $\ell^2(U)$. In a similar fashion, we find that
 \[
\mathbb{E}\,\Psi_U(x)\Psi_U(y)=\sum_j \psi_j(x)\psi_j(y)=\frac{1}{\pi(y)}G_U(x,y),
\] 
where the last equality uses the orthonormal basis property of $\{\psi_j\}_j$ in
the discrete Sobolev space $H^1_0(U)$.
\end{proof}

\subsection{Desirable properties of the sought-after operator}

Given $\Gamma_U$ as above, we consider a connected subgraph $\Gamma_W$ of
$\Gamma_U$
with vertex set $W$. For simplicity we assume that $\Gamma_U$ 'properly envelops' 
$\Gamma_W$: 
$W$ and $\Gamma\setminus U$ are disconnected in $\Gamma$, but for every 
$x\in U\setminus W$ there exists an edge $(x,y)$ in $\Gamma_U$ for some 
$y\in W$ (so that $U\setminus W$ 
forms a 'layer' around $W$). The orthonormal basis in $\ell^2(U)$ is taken to be just the 
collection of point masses in vertices of $U$, $\{\phi_n\} := \{\delta_x\}_{x\in U}$. 
We aim to define an operator $\mathbf{Q}_{U,W} = \mathbf{Q}_W$ acting on $\ell^2(U)$ 
such that $\mathbf{Q}_W\Phi_U = \Psi_W$ holds in the sense of having equal 
probability distributions. 
Now, there happen to be a whole lot of such possible operators, so we wish to narrow our 
search by requiring it to satisfy a certain set of additional conditions:
\begin{itemize}[leftmargin=0.7cm]
\item[$(a)$] \textbf{Isometry.} The operator $\mathbf{Q}_W$ should map the basis 
$\{\delta_x\}_{x\in W}$ to some orthonormal basis $\{\psi_x\}_{x\in W}$ of $H^1_0(W)$.
\item[$(b)$] \textbf{Boundary values.} Since $\Psi_W$ vanishes outside of $W$, so that 
$\supp \mathbf{Q}_Wf \subset W$, it is convenient to assume that $\mathbf{Q}_W$ 
actually does not take values of $f$ outside $W$ into account, so that
\begin{equation}  
\label{e:qp1}
\mathbf{Q}_Wf\equiv 0,\quad \text{if}\,\,\,\supp f\subset U\setminus W.
\end{equation}
\item[$(c)$] \textbf{Markov property.} Let $W'\subset W$ be 'enveloped' by $W$ 
in the above sense. The conditional expectation
\begin{equation}\notag
\varphi_{W,W'} := \mathbb{E}(\Psi_W\vert\Psi_W\mathbbm{1}_{W\setminus W'}),
\end{equation}
which is random because it is conditioned on random values in $W\setminus W'$,
is almost surely discrete harmonic on $W'$. Its boundary values are those of 
$\Psi_W$ on $W\setminus W'$, and we also have that 
$
\Psi_W-\varphi_{W,W'} = \Psi_{W'}$ holds in distribution. This suggests that we should
assume $\mathbf{Q}_{W}f - \mathbf{Q}_{W'}f$ to be discrete harmonic on $W'$ with 
boundary values $\mathbf{Q}_{W}f$. Moreover, by the previous condition, if $f$ is 
supported on $W\setminus W'$, then 
$\mathbf{Q}_{W}f - \mathbf{Q}_{W'}f = \mathbf{Q}_Wf$, 
in particular
\begin{equation}\label{e:qp2}
\mathbf{Q}_W\delta_{x} \textup{ is discrete harmonic on}\; W',\quad x\in W\setminus W'.
\end{equation}
For point masses in $W'$ we actually ask even more, namely stability, 
\begin{equation}
\label{e:qp3}
\mathbf{Q}_W\delta_{x} = \mathbf{Q}_{W'}\delta_x, \quad x\in W',
\end{equation}
so that the difference vanishes: $\mathbf{Q}_{W}\delta_x - \mathbf{Q}_{W'}\delta_x=0$  
(and in particular it is discrete harmonic). 
\end{itemize}
It turns out that there is natural way to construct such an operator -- by growing it from 
a given point. Although the detailed construction is presented in the next section, 
we supply a general outline below.
We assume that the operator $\mathbf{Q}_{W'}$ is given for some subset of vertices 
$W'\subset W$, and that it supplies an isometry 
$\ell^2(W')\rightarrow H^1_0(W')$, i.e. $\mathbf{Q}_{W'}\delta_x = \psi_x,\; x\in W'$, 
for some choice of the orthonormal basis $\{\psi_x\}_{x\in W'}$ of $H^1_0(W')$. 
As usual, we extend $\psi_x(y):= 0$ for $x\in W'$ and $y\in W\setminus W'$.  

Next, we consider the collection of Kronecker deltas 
$\{\delta_{\xi}\}_{\xi\in W\setminus W'}$ which together with the basis of $\ell^2(W')$ 
supplies an orthonormal basis for $\ell^2(W)$. 
We aim to define the action of $\mathbf{Q}_W$ on these functions by putting 
$\mathbf{Q}_W\delta_\xi=\psi_\xi$ for $\xi\in W\setminus W'$, where the functions
$\psi_{\xi}$ for $\xi\in W\setminus W'$ are unit vectors that are mutually orthogonal 
in $H^1_0(W)\ominus H^1_0(W')$, where $H^1_0(W')$ is -- as usual -- thought 
of as a subspace of $H^1_0(W)$ by extending the functions to vanish in 
$W\setminus W'$. 
The thus extended basis $\{\psi_x\}_{x\in W}$ will naturally form an orthonormal basis 
of $H^1_0(W)$, which in turn allows us to go to the next step in the iteration procedure. 
The Poisson kernel for $W$,
\begin{equation}
\tilde{\psi_{\xi}}(y) := 
P_{W, W\setminus W'}(y,\xi)
\label{eq:psitilde-1}
\end{equation}
where $\xi\in W\setminus W'$,  are linearly independent, and moreover,
orthogonal to the subspace $H^1_0(W')=\spn\{\psi_{x}:\,x\in W'\}$. Indeed, we calculate that
\begin{equation} \notag
\begin{split}
& \langle\tilde{\psi}_{\xi},\psi_x\rangle_{\nabla, W} = 
\langle \mathbf{d}^*\mathbf{d}\,\tilde{\psi}_{\xi},\psi_{x}\rangle_{W} = 0,\quad 
\xi\in W\setminus W',\; x\in W',
\end{split}
\end{equation}
since, as a matter of definition, $\mathbf{d}^*\mathbf{d}\,\tilde{\psi}_{\xi}=0$ holds on 
$\supp \psi_x\subset W'$. While the functions $\tilde{\psi}_{\xi},\; \xi\in W\setminus W'$ 
are not, generally speaking, orthogonal to each other, one can easily make them orthogonal, 
via an orthogonalization procedure such as Gram-Schmidt. This means that
there exists a weight function $\rho: W\setminus W' \times W \setminus W' \to \mathbb{R}$ 
such that the functions 
\begin{equation}\notag
\psi_{\xi}(y) := \sum_{\eta}\tilde{\psi}_{\eta}(y)\rho(\eta,\xi), \quad \xi\in 
W\setminus W',\; y\in W,
\end{equation}
form an orthonormal system in $H^1_0(W)$, and, combined with the system 
$\{\psi_{x}\}_{x\in W'}$, form an orthonormal basis for the space $H^1_0(W)$. 
Alternatively, if we make the minimal iteration step of adding a single point to 
$W\setminus W'$ at a time, no additional orthogonality condition is required, and we 
would just normalize $\tilde\psi_\xi$:
$\psi_\xi=\|\tilde\psi_\xi\|_{\nabla,W}^{-1}\tilde\psi_\xi$ for $\xi\in W\setminus W'$.


\section{Discrete Hadamard's formula and the related operator}\label{S:DHF}

\subsection{Hadamard increments for nested vertex sets}
The basic factorization applies to arbitrary nested finite vertex
sets. In particular, it does not require a prescribed discrete
foliation. We formulate it for the inverse Dirichlet Laplacian,
which is the normalized Green operator in the notation of
Subsection~2.6.

\begin{lemma}[Hadamard increment for nested vertex sets]
\label{lem:nested-hadamard}
Let $U\subset W$ be finite vertex sets in a connected locally
finite graph with positive symmetric conductances. Denote by $\Delta_U$ and $\Delta_W$ the Dirichlet graph
Laplacians, and extend $\tilde{\mathbf G}_U:=\Delta_U^{-1}$
by zero to $\ell^2(W)$. Put $B:=W\setminus U$.
Let $P:\ell^2(B)\to\ell^2(W)$ be harmonic extension from
$B$ into $U$, with the values on $B$ prescribed. Then
\[
 \tilde{\mathbf G}_W-\tilde{\mathbf G}_U=P\,\tilde{\mathbf G}_W[B,B]\,P^*
 =K K^*,\qquad K:=P\,\tilde{\mathbf G}_W[B,B]^{1/2},
\]
where $\tilde{\mathbf G}_W[B,B]$ denotes the compression of $\tilde{\mathbf G}_W$ to $B$.
\end{lemma}

\begin{proof}
In the decomposition $\ell^2(W)=\ell^2(U)\oplus\ell^2(B)$,
write
\[
 \Delta_W=\begin{pmatrix}A&C\\ C^*&D\end{pmatrix},
 \qquad A=\Delta_U.
\]
The Schur complement $S:=D-C^*A^{-1}C$ is positive
definite, and the harmonic extension operator is
\[
 P=\begin{pmatrix}-A^{-1}C\\ I_B\end{pmatrix}.
\]
The block inverse formula gives
\[
 \Delta_W^{-1}=
 \begin{pmatrix}A^{-1}&0\\0&0\end{pmatrix}
 +P S^{-1}P^*.
\]
Its lower-right block is $S^{-1}=\tilde{\mathbf G}_W[B,B]$.
Taking the positive square root proves the assertion.
\end{proof}

The successive growth clusters of a discrete foliation form
nested vertex sets. Lemma~\ref{lem:nested-hadamard} therefore
provides the positive increment identity underlying the
layerwise formula below, which makes the corresponding
kernels explicit.

\subsection{Hadamard variation formula in the discrete setting}

We now proceed to derive the discrete version of the Hadamard's variational formula. 
To prepare the ground for that, we define the boundary normalized Green operator 
$\tilde{\mathbf{G}}^{\langle n\rangle}:\ell^2(V(\gamma_n))\to \ell^2(V(\gamma_n))$ by  
\begin{equation}
\label{eq:Green-bdry-1}
\tilde{\mathbf{G}}^{\langle n\rangle}f(x)=
\sum_{y\in V(\gamma_n)}\tilde G_n(x,y)f(y)=
\sum_{y\in V(\gamma_n)}G_n(x,y)\frac{f(y)}{\pi(y)},\qquad x\in V(\gamma_n).
\end{equation}
Since $\tilde{G}_n(x,y)$ is symmetric and defines a positive operator 
$\tilde{\mathbf{G}}_n$ on $\ell^2(V(\Gamma_n))$, the boundary operator 
$\tilde{\mathbf{G}}^{\langle n\rangle}$ is definitely symmetric and semi-positive. 
Later on, we will need to know that it is actually positive, and hence invertible.

\begin{proposition}
The boundary Green operator 
$\tilde{\mathbf{G}}^{\langle n\rangle}:\ell^2(V(\gamma_n))\to \ell^2(V(\gamma_n))$
given by \eqref{eq:Green-bdry-1} is symmetric with positive eigenvalues.
In particular, $\tilde{\mathbf{G}}^{\langle n\rangle}$ is invertible. 
\label{prop:Green-inv-1}
\end{proposition}

\begin{proof}
The operator $\tilde{\mathbf G}^{\langle n\rangle}$ is the
compression of the positive definite operator
$\tilde{\mathbf G}_n=\Delta_n^{-1}$ to the nonzero
subspace $\ell^2(V(\gamma_n))$. Hence it is positive definite.
\end{proof}

Define the positive square root
$\tilde{\mathbf R}_n:=(\tilde{\mathbf G}^{\langle n\rangle})^{1/2}$
and write $\tilde R_n(x,y)$ for its symmetric kernel.
We are now ready to formulate the general discrete Hadamard variational formula.

\begin{theorem}{{\rm (Discrete Hadamard variational formula)}}
\label{Hadamard}
For $x,y \in V$, we have
\begin{equation}
\tilde{G}_n(x,y) = \frac{G_{n}(x,y)}{\pi(y)}=
\sum_{\xi\in V(\Gamma_{n})}K_{t(\xi)}(x,\xi) K_{t(\xi)}(y,\xi),
\end{equation}
where $t(\xi)$ denotes the unique discrete time step $m\in\mathbb{Z}_{\geq 0}$ 
such that $\xi\in V(\gamma_m)$, and the indicated kernel is given by
\begin{equation}
K_{m}(y,\xi):=\sum_{\eta\in V(\gamma_m)}P_m(y,\eta)\tilde{R}_m(\eta,\xi),\qquad    
\xi \in V(\gamma_m).
\end{equation}
\end{theorem}

\proof[Proof]
Apply Lemma~\ref{lem:nested-hadamard} to the nested sets
$V(\Gamma_{m-1})\subset V(\Gamma_m)$, with new layer
$V(\gamma_m)$. The harmonic extension is $\mathbf P_m$,
while the compressed Green operator is
$\tilde{\mathbf G}^{\langle m\rangle}$. Consequently,
\[
\tilde{\mathbf G}_m-\tilde{\mathbf G}_{m-1}
=\mathbf P_m(\tilde{\mathbf R}_m)^2\mathbf P_m^*
=\mathbf K_m\mathbf K_m^*,
\]
where $\mathbf K_m=\mathbf P_m\tilde{\mathbf R}_m$ has
precisely the kernel stated in the theorem. For $m=0$,
use the same factorization with the empty initial vertex set,
so that $\tilde{\mathbf G}_0=\mathbf K_0\mathbf K_0^*$.
Summing over $m=0,\ldots,n$ proves the formula.
\endproof

\subsection{Definition of the Hadamard operator}

\noindent We now introduce the \emph{Hadamard operator} 
$\mathbf{Q}_n:\ell^2(V(\Gamma_n))\to \ell^2(V(\Gamma_n))$ as given by the 
formula
\begin{equation}
\mathbf{Q}_nf(x):=\sum_{y\in V(\Gamma_n)}K_{t(y)}(x,y)f(y).
\end{equation}
Its adjoint with respect to the standard $\ell^2(V(\Gamma_n))$-duality is then
\begin{equation}
\mathbf{Q}_n^*f(x)=
\sum_{y\in V(\Gamma_n)}{K_{t(x)}(y,x)}f(y),\qquad x\in V(\Gamma_n).
\end{equation}
We note that for $a\in V(\Gamma_n)$, a calculation gives that
\begin{equation}
\mathbf{Q}_n\mathbf{Q}_n^*\delta_a(z)=\sum_{\xi\in V(\Gamma_n)}
K_{t(\xi)}(z,\xi)K_{t(\xi)}(a,\xi)=\frac{G_n(z,a)}{\pi(a)}=
\tilde{G}_n(z,a)=\tilde{\mathbf{G}}_n\delta_a(z), \quad a\in V(\Gamma_n),
\end{equation}
where the middle identity relies on the discrete Hadamard variation formula
of Theorem \ref{Hadamard}.
Since $a\in V(\Gamma_n)$ is arbitrary, it is now immediate that 
$\mathbf{Q}_n\mathbf{Q}^*_n=\tilde{\mathbf{G}}_n$. This property is a
basic ingredient in the proof of the following theorem.

\begin{theorem}
The Hadamard operator $\mathbf{Q}_n$ is an isometric isomorphism 
$\ell^2(V(\Gamma_n))\to H_0^1(V(\Gamma_n))$.
\label{thm:discreteHN}
\end{theorem}

\begin{proof}
We recall the notational convention 
$\Delta_{n}:=\Delta_{V(\Gamma_n)}$, and observe that
there is a well-defined canonical square root $(\Delta_n)^{\frac12}$ 
which acts as a symmetric matrix with real entries on the finite-dimensional 
space $\ell^2(V(\Gamma_n))$.
For $f,g \in \ell^2(V(\Gamma_n))$, we check that
\begin{multline}
\big\langle \mathbf{Q}_n^*(\Delta_n)^\frac{1}{2}f,
\mathbf{Q}_n^*(\Delta_n)^\frac{1}{2}g\big\rangle_{V(\Gamma_n)}
=\big\langle \mathbf{Q}_n\mathbf{Q}_n^*(\Delta_n)^\frac{1}{2}f,
(\Delta_n)^\frac{1}{2}g\big\rangle_{V(\Gamma_n)}
\\
=\big\langle \tilde{\mathbf{G}}_n(\Delta_n)^\frac{1}{2}f,
(\Delta_n)^\frac{1}{2}g\big\rangle_{V(\Gamma_n)}
=\big\langle(\Delta_n)^\frac{1}{2}
\tilde{\mathbf{G}}_n(\Delta_n)^\frac{1}{2}f,g
\big\rangle_{V(\Gamma_n)}
=\langle f,g \rangle_{V(\Gamma_n)},
\end{multline}
since $\tilde{\mathbf{G}}_n=\Delta^{-1}_n$ holds.
This entails that the operator 
$\mathbf{Q}_n^*(\Delta_n)^\frac{1}{2}:\ell^2(V(\Gamma_n)) 
\to \ell^2(V(\Gamma_n))$ is an isometry, which immediately entails that the 
adjoint operator 
$(\Delta_n)^\frac{1}{2} \mathbf{Q}_n:\ell^2(V(\Gamma_n)) 
\to \ell^2(V(\Gamma_n))$ is a partial isometry, and, in particular, a contraction. 
Suppose we can show that the kernel (or null space) of 
$(\Delta_n)^\frac{1}{2} \mathbf{Q}_n$ is trivial. Then the operator
$(\Delta_n)^\frac{1}{2} \mathbf{Q}_n$ must be an isometry, in which case
\begin{equation}
\langle \mathbf{Q}_nf,\mathbf{Q}_ng\rangle_{\nabla,\Gamma_n}
=\big\langle(\Delta_n)^{\frac{1}{2}}\mathbf{Q}_nf,
(\Delta_n)^{\frac{1}{2}}\mathbf{Q}_ng\big\rangle_{\Gamma_n}
=\langle f,g\rangle_{\Gamma_n}.
\end{equation}
In other words, we would have shown that the operator 
$\mathbf{Q}_n:\ell^2(\Gamma_n)\to H_0^1(\Gamma_n)$ defines an isometric 
isomorphism, as claimed. We assume that 
$f_0\in\ell^2(V(\Gamma_n))$ with 
\[
(\Delta_n)^\frac{1}{2} \mathbf{Q}_n f_0=0,
\]
and intend to show that $f_0=0$.
Next, since $\Delta_n$ is invertible on $\ell^2(V(\Gamma_n))$, 
we know that $(\Delta_n)^{\frac{1}{2}}$ is invertible too, and,
consequently, that $\mathbf{Q}_n f_0=0$.  From the definition of the Hadamard
operator, we see that
\[
\mathbf{Q}_n f_0(x)=\sum_{y\in V(\Gamma_n)}K_{t(y)}(x,y)f_0(y)=
\sum_{y\in V(\gamma_n)}K_{n}(x,y)f_0(y),\qquad x\in V(\gamma_n),
\]
since $K_m(x,y)=0$ if $x\in V(\gamma_n)$ and 
$y\in V(\gamma_m)$ with $m<n$.
Moreover, 
\[
K_n(x,y)=\sum_{\eta\in V(\gamma_n)}P_n(x,\eta)\tilde{R}_n(\eta,y)=
\tilde{R}_n(x,y),\qquad x,y\in V(\gamma_n),
\]
since the Poisson kernel at the discrete boundary is rather trivial: 
$P_n(x,\eta)=\delta_x(\eta)$ for $x,\eta\in V(\gamma_n)$.
As a consequence, we have that
\[
\mathbf{Q}_n f_0(x)=
\sum_{y\in V(\gamma_n)}\tilde{R}_{n}(x,y)f_0(y)
=\tilde{\mathbf{R}}_n[f^{\langle n\rangle}_0](x),\qquad x\in V(\gamma_n),
\]
where we use the notation $f^{\langle n\rangle}_0$ to denote the restriction
of $f_0$ to the foliation layer $V(\gamma_n)$. From our assumption we know
that $\mathbf{Q}_n f_0=0$, and hence it follows that 
$\tilde{\mathbf{R}}_n[f^{\langle n\rangle}_0]=0$. Next, since 
$\tilde{\mathbf{R}}_n$ is invertible in view of Proposition 
\ref{prop:Green-inv-1} and the definition of $\tilde{\mathbf{R}}_n$ as the
canonical self-adjoint square root of $\tilde{\mathbf{G}}^{\langle n\rangle}$,
it is immediate that $f^{\langle n\rangle}_0=0$, which means that $\supp f_0
\subset V(\Gamma_{n-1})$. But then we may observe that 
\[
\mathbf{Q}_n f_0(x)=\sum_{y\in V(\Gamma_{n})}K_{t(y)}(x,y)f_0(y)=
\sum_{y\in V(\gamma_{n-1})}K_{n-1}(x,y)f(y),\qquad x\in V(\gamma_{n-1}),
\]
which puts us in the analogous situation in the preceding foliation layer
$\gamma_{n-1}$. If we proceed we the same arguments as for the layer 
$\gamma_n$, we obtain that $\supp f_0\subset V(\Gamma_{n-2})$. 
By iteration $n$ steps, we obtain that $\supp f_0\subset V(\gamma_0)$, 
and if we repeat the same argument once more, we find that 
$\tilde{\mathbf{R}}_n f_0^{\langle 0\rangle}]=0$ which in turn gives  
$f_0^{\langle 0\rangle}=0$ and hence that $f_0=0$ identically. 
The proof of the theorem is now complete.
\end{proof}

\section{Growing the Discrete Gaussian Free Field
using white noise as building blocks}\label{S:GDGFF}

\subsection{Formulation of the main result}

In this Section we formulate and prove our main result, which is analogous to 
\cite[Theorem 5.1]{HN}. 
In terms of notation, we write 
\[
\mathbf{Q}_nf = \mathbf{Q}_n[\mathds{1}_{V(\Gamma_n)}f]
\]
for a function $f:V\to\mathbb{R}$ not necessarily supported on $V(\Gamma_n)$. 

\begin{theorem}\label{t:main51}
Let $\Phi$ denote the White Noise Field on the graph $\Gamma_N$, where 
$N=0,1,2,\ldots$ is given. For $n=0,1,\ldots,N$, we define the stochastic fields 
 $\Psi_n=\mathbf{Q}_n\Phi$. Then $\Psi_n$ is the Gaussian Free Field on 
 $\Gamma_n$. Moreover,  the increments of the process $\Psi_n$ with respect to 
 the parameter $n$ are independent. 
 More precisely, if $n_i\in\mathbb{Z}_{\ge0}$ are listed increasingly with
 $0\leq n_0<...< n_k\leq N$,
 then the $k$ random vectors 
 $(\Psi_{n_{i}}(x)-\Psi_{n_{i-1}}(x))_{x\in V(\Gamma_N)}$
 for $i=1,\ldots,k$ are all stochastically independent.   
\end{theorem}

\proof[Proof]
Since the distribution of $\Phi$ does not depend on the orthonormal basis 
of $\ell^2(\Gamma_N)$ that defines the field, we may as well assume use the
Kronecker delta basis, so that
\begin{equation}
 \Phi(x)=\sum_{y \in V(\Gamma_N)}\xi_y\, \delta_y(x),\qquad x\in V(\Gamma_N),
\end{equation} 
where, as usual, $\{\xi_y\}_{y \in V(\Gamma_{N})}$ are independent standard 
real Gaussian random variables. It then follows that 
\begin{equation}
\Psi_n(x)= \mathbf{Q}_n\Phi(x)=\sum_{y \in V(\Gamma_N)}\xi_y\,
 \mathbf{Q}_n\delta_y(x)
 ,\qquad x\in V(\Gamma_N).
\end{equation}
According to Theorem \ref{thm:discreteHN}, 
$\mathbf{Q}_n:\ell^2(V(\Gamma_n)) \to H_0^1(V(\Gamma_n))$ is an isometric 
isomorphism, and, consequently,
the set $\{\mathbf{Q}_n\delta_y\}_{y\in V(\Gamma_n)}$ forms an orthonormal 
basis of $H_0^1(\Gamma_n)$. So, $\Psi_n$ is indeed the Gaussian Free Field 
on $\Gamma_n$.
Next, as to the increments, we have 
\begin{equation}
 \begin{gathered}
\Psi_{n_{i}}(x)-\Psi_{n_{i-1}}(x)=\sum_{y\in V(\Gamma_{n_{i}})\setminus 
V(\Gamma_{n_{i-1}})} \xi_y\, \mathbf{Q}_{n_{i}}\delta_y(x),
\end{gathered}
\end{equation}
as it is clear from the definition of $Q_n$ that for $0\leq k_1<k_2$ and 
$f\in \ell^2(V(\Gamma_{k_1}))\subset\ell^2(V(\Gamma_{k_2}))$, we have 
$\mathbf{Q}_{k_1}f=\mathbf{Q}_{k_2}f$. 
The required independence follows from the stochastic independence of the
individual Gaussians in the family $\{\xi_y\}_y$. 
\endproof

\section{Properties of the constructed Gaussian free field $\Psi_n$}
\label{S:PDGFF}

\subsection{Harmonically extended weighted white noise as time-step increments}

In this Section we discuss some interesting properties of the Gaussian Free Field 
constructed in Theorem \ref{t:main51}. 
We first consider the discrete time-derivative of $\Psi_n$, the Gaussian Free Field 
on the foliation cluster $\Gamma_n$ produced from the White Noise $\Phi$ on the 
bigger graph $\Gamma_N$,  $0<n\leq N$. For any $x\in V$, we have that 
\begin{multline}
\notag
 \Psi_{n}(x)-\Psi_{n-1}(x)=\mathbf{Q}_{n}\Phi(x)-\mathbf{Q}_{n-1}\Phi(x)=
 \sum_{y\in V(\gamma_n)}K_n(x,y)\Phi(y)
 \\
 =\sum_{y\in V(\gamma_n)}
 \sum_{\eta\in V(\gamma_n)}P_n(x,\eta)\tilde{R}_n(\eta,y)\Phi(y)=
 \sum_{\eta\in V(\gamma_n)}P_n(x,\eta) \sum_{y\in V(\gamma_n)}
 \tilde{R}_n(\eta,y)\Phi(y).
\end{multline}
While the left-hand side expresses the discrete time derivative of $\Psi_n$, the
right-hand side can be understood as the Poisson extension of a weighted 
White Noise Field on the foliation layer $\gamma_n$. 
Note that this result finds its two-dimensional smooth continuous analogue in 
\cite{HN}. Also, it is not hard to see that a similar computation for the difference 
$\Psi_n -\Psi_m$ with $0\leq m <n \leq N$ gives the well known Markov 
property of the Gaussian Free Field.

\subsection{The connection with Brownian motion}
We turn to another property of the constructed $\GFF$. First, we fix a vector 
$f\in \ell^2(\Gamma_N)$. Then the random vector 
$\left(\langle f,\Psi_n\rangle_{\Gamma_n}\right)_n$ where $n=0,\ldots,N$ has 
a Gaussian distribution (meaning that the entries are jointly Gaussian). 
We will compute its distribution explicitly, moreover we show that it is the same 
as that of a weighted one-dimensional random walk (cf. the corresponding 
result in \cite{HN}).

\begin{proposition}
\label{PropertyGFF} 
For $f\in\ell^2(V(\Gamma_N))$, we 
have that
\begin{equation}
\big(\langle f,\Psi_n\rangle_{V(\Gamma_n)}\big)_{n} \overset{d}
=\Big( B\big(\|\mathbf{Q}^*_nf\|^2_{V(\Gamma_n)}\big) 
\Big)_{n} \overset{d}=\bigg(\sum_{k=0}^n 
\big(B(k+1)-B(k) \big) \|\mathds{1}_{V(\gamma_k)}
\mathbf{Q}^*_n f\|_{V(\Gamma_n)}\bigg)_{n}
\end{equation}
where $n$ ranges over $\{0,\ldots,N\}$ and "$\overset{d}=$" means equality in 
probability distribution, while $B(\cdot)$ stands for a standard continuous Brownian 
motion on $\mathbb{R}_{\ge0}$ with initial value $B(0)=0$.
\end{proposition}

\proof[Proof]
For brevity of notation, we write 
\[
F_n(f):=\langle f,\Psi_n\rangle_{V(\Gamma_n)}=\sum_{y\in V(\Gamma_n)}
\xi_y\,(\mathbf{Q}_n^*f)(y),
\] 
where we recall that the $\xi_y$ are all independent standard Gaussian random
variables. We find that 
\[
\mathbb{E}\,F_n(f)=\sum_{y\in V(\Gamma_n)}
(\mathbf{Q}_n^*f)(y)\,\mathbb{E}\,\xi_y=0
\]
and 
\[
F_n(f)=\langle f,\Psi_n\rangle_{V(\Gamma_n)}
=\langle f,\mathbf{Q}_n\Phi\rangle_{V(\Gamma_n)}
=\langle\mathbf{Q}^*_nf, 
\Phi\rangle_{V(\Gamma_N)}.
\]
In addition to the first moments, we shall need the second moments as well.
For $n\leq m \leq N$, we calculate that
\begin{multline}
\notag
\mathbb{E}(F_n(f)F_m(f))=
\mathbb{E}\big(\langle\mathbf{Q}^*_nf, 
\Phi \rangle_{V(\Gamma_N)}\langle\mathbf{Q}^*_mf, 
\Phi \rangle_{V(\Gamma_N)}\big)
\\
=\langle \mathbf{Q}^*_nf,\mathbf{Q}^*_m f\rangle_{V(\Gamma_N)} 
= \sum_{x\in V(\Gamma_N)}\mathbf{Q}^*_nf(x)\mathbf{Q}^*_mf(x).
\end{multline}
We note that $\mathbf{Q}^\ast_n f=0$ holds on 
$V(\Gamma_N)\setminus V(\Gamma_n)$, while $\mathbf{Q}^\ast_n f=
\mathbf{Q}^\ast_m f$ holds on $V(\Gamma_n)$. 
so that by the above calculation, it follows that 
\begin{equation}
\mathbb{E}(F_n(f)F_m(f))=\langle\mathbf{Q}^*_nf,\mathbf{Q}^*_n f
\rangle_{\Gamma_N}
=\|\mathbf{Q}^\ast_n f\|^2_{V(\Gamma_N)}=
\|\mathbf{Q}^\ast_n f\|^2_{V(\Gamma_n)},\qquad n\le m\le N.
\end{equation}
The first and second moment structure of a standard Brownian motion $B(t)$ 
with $B(0)=0$ is 
\[
\mathbb{E}\,B(t)=0\,\,\,\text{ and }\,\,\,\mathbb{E}\,(B(t)B(t'))=\min\{t,t'\}
\]
for $t,t'\in\mathbb{R}_{\ge0}$. This means that 
$\big(B(\|\mathbf{Q}^*_nf\|^2_{V(\Gamma_n)})\big)_{n}$ has the same first and
second moments (including covariances) as $(F_n(f))_n$, so that by (joint) normality,
the first equality in distribution follows. 
Next, we put
\begin{equation}
W_n(f):=\sum_{k=0}^n\big(B(k+1)-B(k)\big)\|
\mathds{1}_{V(\gamma_k)}\mathbf{Q}^*_n f\big\|_{V(\Gamma_n)}. 
\end{equation}
Then $\mathbb{E}W_n(f)=0$, and regarding the second moments, we find that
for $n\leq m \leq N$, 
\begin{multline}
 \mathbb{E}( W_n(f)W_m(f)) = 
 \sum_{k=0}^n 
 \big\|\mathds{1}_{V(\gamma_k)}\mathbf{Q}^*_n f\big\|_{V(\Gamma_n)}
 \|\mathds{1}_{V(\gamma_k)}\mathbf{Q}^*_m f\big\|_{V(\Gamma_m)}
\\
=\sum_{k=0}^{n} \big\|\mathds{1}_{V(\gamma_k)}\mathbf{Q}^*_n f
\big\|_{V(\Gamma_n)}^2
= \big\|\mathbf{Q}^*_n f\big\|^2_{V(\Gamma_n)}.
\end{multline}
The second equality in distribution now follows as well.
\endproof

We recall the Poisson extension operator 
$\mathbf{P}_n: \ell^2(\gamma_n) \to \ell^2(\Gamma_n)$ given by
\[
\mathbf{P}_nf(x)=\sum_{y\in V(\gamma_n)}P_n(x,y)f(y),\qquad x\in\Gamma_n,
\]
where $P_n=P_{V(\Gamma_n),V(\gamma_n)}$, and denote by 
$\mathbf{P}_n^*$ its adjoint operator $\ell^2(\Gamma_n) \to \ell^2(\gamma_n)$.
It is is sometimes called the \emph{sweep operator}, at least in the continuous context, 
and it is given by the formula
\[
\mathbf{P}_n^\ast f(y)=\sum_{x\in V(\Gamma_n)}P_n(x,y)f(x),
\qquad y\in\gamma_n.
\] 
We are now ready to state a result concerning the behavior of the boundary averages of the 
$\GFF$ process, analogous to what was obtained in the two-dimensional continuous smooth
setting in \cite{HN}. 

\begin{proposition}
Suppose $n_1,n_2\in \{1,\ldots,N\}$ and that $n_1<n_2$. Assume, moreover, that
$f\in\ell^2(\Gamma_{n_1})\subset\ell^2(\Gamma_N)$. 
We then have that
\begin{multline}
\notag
\big(\langle\mathbf{P}_n^*f,\Psi_{n_2}\rangle_{\ell^2(\gamma_n)}  \big)_{n} 
\overset{d}=\Big(B\Big(\big\|\mathds{1}_{V(\Gamma_{n_2})\setminus V(\Gamma_n)}
\mathbf{Q}_{n_2}^\ast f\big\|^2_{V(\Gamma_{n_2})}\Big)\Big)_{n}
\\
\overset{d}
= \Bigg( \sum_{k=n}^{n_2}\big(B(k+1)-B(k)\big)\big\|\mathds{1}_{V(\gamma_{k})}
\mathbf{Q}^*_{n_2}f\big\|_{V(\Gamma_{n_2})}\big)\Bigg)_{n},
\end{multline}
where $n$ ranges over over $n=n_1,\ldots,n_2$. Here, $B(\cdot)$ denotes a standard 
continuous Brownian motion on $\mathbb{R}_{\ge0}$ starting at $B(0)=0$.
\end{proposition}

\proof[Proof]
We write 
\[
A_n(f):=\langle\mathbf{P}_n^*f,\Psi_{n_2}\rangle_{V(\gamma_n)}=
\langle f,\mathbf{P}_n\Psi_{n_2}\rangle_{V(\Gamma_n)},
\]
where it is understood that $\mathbf{P}_n\Psi_{n_2}$ vanishes on 
$V(\Gamma_N)\setminus V(\Gamma_n)$. We identify $\mathbf{P}_n\Psi_{n_2}$ with
$\Psi_{n_2}-\Psi_{n-1}$ on $V(\Gamma_n)$, simply because it is harmonic on 
$V(\Gamma_{n-1})$ and has the same value on the boundary layer $V(\gamma_n)$. 
It now follows that
\[
A_n(f)=\big\langle f, \mathds{1}_{V(\Gamma_{n_1})}(\Psi_{n_2}-\Psi_{n-1})
\big\rangle_{V(\Gamma_{n_1})}=
\langle f, \Psi_{n_2}\rangle_{V(\Gamma_{n_2})}-
\langle f,\Psi_{n-1}\rangle_{V(\Gamma_{n-1})}=F_{n_2}(f)-F_{n-1}(f),
\]
where the notation $F_k(f)$ is as in the proof of Proposition \ref{PropertyGFF}. 
The first moments vanish: $\mathbb{E}(A_n(f))=0$. We proceed to consider the 
second moments. The second moments of $F_k(f)$ are known from the proof of
Proposition \ref{PropertyGFF}, and we can use them to obtain the second moments 
of $A_n(f)$. 
The result is that if $n\le m$,
\[
\mathbb{E}(A_n(f)A_m(f))=\|\mathbf{Q}_{n_2}^\ast f\|^2_{V(\Gamma_{n_2})}
-\|\mathbf{Q}_{m-1}^\ast f\|^2_{V(\Gamma_{m-1})}.
\]
We compare these first and second moments with those obtained for
the other two processes involving Brownian motion. The details are analogous
with the calculations in the proof of Proposition \ref{PropertyGFF} and for this reason
omitted. 
\endproof

\section{Scaling limit of the growing field on the square lattice}
\label{S:SCDGFF}

We first establish convergence of the discrete Dirichlet Green
operators, then identify the scaling limit of the growing Gaussian
field. Finally, we show that the integrated Hadamard increments
converge as Green-operator forms.

Let $\{\Omega_t:0<t\leq1\}$ be an increasing family of bounded
planar domains as in~\cite{HN}, with smooth boundaries and
$\overline{\Omega_s}\subset\Omega_t$ whenever $s<t$.
Fix $0<t_1<\cdots<t_m\leq1$.
For each integer $n$, choose nested finite lattice sets
\[
 U_{n,1}\subset\cdots\subset U_{n,m}\subset\mathbb Z^2
\]
which are admissible approximations of
$\Omega_{t_1},\ldots,\Omega_{t_m}$, respectively.
For example, one may use the interior-cell convention
\[
 U_{n,k}:=\left\{x\in\mathbb Z^2:
 x/n+[0,n^{-1}]^2\subset\Omega_{t_k},
 \quad\operatorname{dist}_{\infty}
 (x/n,\mathbb C\setminus\Omega_{t_k})>n^{-1}\right\}.
\]
These sets are nested. For each fixed $k$, every compact subset
of $\Omega_{t_k}$ is eventually covered by their associated
cells, while those cells lie inside $\Omega_{t_k}$. Any empty or
disconnected sets at finitely many small values of $n$ may be
discarded. The nested-set Hadamard factorization of
Lemma~\ref{lem:nested-hadamard} applies directly to these sets.

Write $\tilde{\mathbf G}_{n,k}$ for the inverse of the
Dirichlet combinatorial Laplacian on $U_{n,k}$, extended by zero
to a common ambient lattice space. The conductances on the
nearest-neighbor edges are taken to be $1$.
Let $\Psi_{n,k}$ denote the jointly constructed centered
Gaussian fields obtained by applying
Lemma~\ref{lem:nested-hadamard} successively to the nested
sets, using independent standard Gaussian noise on each new
set of vertices. Their covariance satisfies
\[
 \operatorname{Cov}(\Psi_{n,k}(x),\Psi_{n,\ell}(y))
 =\tilde G_{n,\min(k,\ell)}(x,y).
\]
For a bounded test function $f$ on $\Omega_1$, extended by zero
outside, put
\[
 X_{n,k}(f):=n^{-2}\sum_{x\in U_{n,k}}f(x/n)\Psi_{n,k}(x).
\]
For definiteness in the point-sampling formula, we initially
take $f$ continuous and compactly supported in $\Omega_1$.

\subsection{Convergence of discrete Green operators}
\label{SS:GreenForms}

We give a variational proof of the Green-form convergence used
above. Throughout this subsection, planar integrals use ordinary
Lebesgue measure. Let $\Omega$ be a bounded $C^1$ domain,
and set $h:=n^{-1}$. Let $U_h$ be the interior-cell approximation
specified above, and extend
lattice functions by zero off $U_h$.

For lattice functions $v,w$, write
\[
 (v,w)_h:=h^2\sum_xv(x)w(x),\qquad
 \mathcal E_h(v,w):=\sum_{j=1}^2\sum_x
 (v(x+e_j)-v(x))(w(x+e_j)-w(x)).
\]
The second sum is over the entire lattice, with zero extension.
In particular,
$\mathcal E_h(v,w)=(h^{-2}\Delta_hv,w)_h$
for the combinatorial Dirichlet Laplacian $\Delta_h$.

\begin{lemma}\label{lem:discrete-resolvent}
Let $f,g\in C_c^\infty(\Omega)$ and let
$u_h=\Delta_h^{-1}f_h$, where $f_h(x)=f(hx)$ on $U_h$.
Then the cellwise-constant interpolation of $h^2u_h$
converges strongly in $L^2(\Omega)$ to the unique weak
solution $u\in H_0^1(\Omega)$ of
$-\Delta u=f$. Consequently,
\[
 h^4\sum_{x,y\in U_h}g(hx)\,
 \tilde G_h(x,y)f(hy)
 \longrightarrow \int_\Omega g(z)u(z)\,\mathrm dz
 =\int_{\Omega^2}g(z)G_\Omega(z,w)f(w)
 \,\mathrm dz\,\mathrm dw.
\]
The conclusion extends to continuous compactly supported
$f,g$ by uniform approximation and the energy bounds below.
\end{lemma}

\begin{proof}
Put $v_h=h^2u_h$. It solves
$h^{-2}\Delta_hv_h=f_h$, or equivalently
\[
 \mathcal E_h(v_h,w_h)=(f_h,w_h)_h
\]
for all lattice functions supported in $U_h$.
The zero-extension discrete Poincar\'e inequality on a fixed
box containing $\Omega$ gives
\[
 \|v_h\|_h^2\leq C\mathcal E_h(v_h,v_h).
\]
Taking $w_h=v_h$ and using Cauchy--Schwarz yields
$\|v_h\|_h+\mathcal E_h(v_h,v_h)^{1/2}\leq C\|f_h\|_h$,
uniformly in $h$.

Let $V_h$ denote the cellwise-constant interpolation of
$v_h$, extended by zero to $\mathbb R^2$. The energy bound
implies the discrete translation estimate
\[
 \|V_h(\,\cdot+a)-V_h\|_{L^2(\mathbb R^2)}
 \leq C(|a|+h)\mathcal E_h(v_h,v_h)^{1/2}.
\]
To see this for arbitrary translations, first consider a residual
coordinate displacement $0\leq\delta\leq h$. The cell convention
$hx+[0,h)^2$ gives the exact identity
\[
 \|V_h(\,\cdot+\delta e_j)-V_h\|_{L^2(\mathbb R^2)}^2
 =\delta h\sum_x|v_h(x+e_j)-v_h(x)|^2.
\]
For an integral lattice displacement $kh e_j$, telescoping the
nearest-neighbor differences and applying Cauchy--Schwarz gives
\[
 \|V_h(\,\cdot+kh e_j)-V_h\|_2
 \leq |k|h\,\mathcal E_h(v_h,v_h)^{1/2}.
\]
Writing each coordinate displacement as an integral multiple of
$h$ plus a residual displacement of magnitude less than $h$,
and applying the triangle inequality in both coordinate
directions, proves the displayed translation estimate.
Together with the fixed compact support, the
Fr\'echet--Kolmogorov criterion gives precompactness in
$L^2(\mathbb R^2)$.

Let $V$ be a subsequential strong limit. Boundedness of the
difference quotients gives weak $L^2$ limits of the discrete
gradients, identified distributionally as $\partial_jV$.
Thus $V\in H^1(\mathbb R^2)$. Since the zero-extended
interpolants vanish outside an $O(h)$ neighborhood of
$\overline\Omega$, $V=0$ almost everywhere off $\Omega$.
For a $C^1$ domain this implies $V|_\Omega\in H_0^1(\Omega)$.

For $\phi\in C_c^\infty(\Omega)$, its lattice samples
$\phi_h(x)=\phi(hx)$ are supported in $U_h$ for sufficiently
small $h$. Discrete summation by parts gives
\[
 \mathcal E_h(v_h,\phi_h)=(f_h,\phi_h)_h.
\]
Passing to the limit, using weak convergence of discrete
gradients and strong convergence of sampled smooth
gradients, gives
\[
 \int_\Omega\nabla V\cdot\nabla\phi\,\mathrm dz
 =\int_\Omega f\phi\,\mathrm dz.
\]
By density, $V$ is the unique weak Dirichlet solution.
Uniqueness identifies every subsequential limit, proving
strong convergence of the full sequence.

Finally,
\[
 h^4\sum_{x,y}g(hx)\tilde G_h(x,y)f(hy)
 =(g_h,v_h)_h
 \longrightarrow\int_\Omega gu\,\mathrm dz.
\]
The same uniform Poincar\'e and energy estimates give
uniform operator bounds from discrete $L^2$ data to
interpolated solutions, allowing approximation of continuous
compactly supported data by smooth functions.
\end{proof}

\subsection{Scaling limit of the growing discrete Gaussian free field}
\label{SS:ScalingLimit}

We now connect the finite-graph Hadamard construction with the
continuum growing Gaussian free field of Hedenmalm and
Nieminen~\cite{HN}. The theorem identifies the finite-dimensional distributions
of the growing fields paired with test functions.

\begin{theorem}\label{t:SL}
For every $k$ and all continuous compactly
supported test functions $f,g$, the discrete Green forms obey
\begin{equation}\label{eq:green-form-convergence}
 n^{-4}\sum_{x,y\in U_{n,k}}
 f(x/n)\tilde G_{n,k}(x,y)g(y/n)
 \longrightarrow
 \int_{\Omega_{t_k}}\int_{\Omega_{t_k}}
 f(z)G_{\Omega_{t_k}}(z,w)g(w)
 \,\mathrm dz\,\mathrm dw,
\end{equation}
where $G_{\Omega}$ is the Dirichlet Green kernel for
$-\Delta$ with respect to ordinary planar Lebesgue measure.
Consequently, for every finite collection of times and test functions
the vector of variables $X_{n,k}(f)$ converges in distribution
to the corresponding evaluations of the growing continuum
Gaussian free field of~\cite{HN}.
\end{theorem}

\begin{proof}
For $k\leq\ell$, the nested-set factorization
of Lemma~\ref{lem:nested-hadamard} and independence
of the added Gaussian layers give
\[
 \operatorname{Cov}(X_{n,k}(f),X_{n,\ell}(g))
 =n^{-4}\sum_{x,y\in U_{n,k}}
 f(x/n)\tilde G_{n,k}(x,y)g(y/n).
\]
The limiting growing continuum field has the corresponding
covariance
\[
 \int_{\Omega_{t_k}}\int_{\Omega_{t_k}}
 f(z)G_{\Omega_{t_k}}(z,w)g(w)
 \,\mathrm dz\,\mathrm dw.
\]
Lemma~\ref{lem:discrete-resolvent} therefore gives
entrywise convergence of the covariance matrices. All
variables are jointly centered Gaussian, so their
finite-dimensional distributions converge.
\end{proof}

\begin{remark}
For standard admissible square-lattice approximations of smooth
planar domains, convergence of Green forms follows from the
classical discrete-to-continuum Dirichlet Green-function
theory. See, for example, the scaling-limit discussion and
Theorem~1.27 of Biskup~\cite{B}. A normalization distinction
is important: the simple-random-walk Green function on the
unit-conductance square lattice is $4$ times the inverse of
the combinatorial Dirichlet Laplacian used here. Accordingly,
a DGFF normalized using the random-walk Green kernel is
twice the field used here. The limiting covariance must be
adjusted by the same factor. Equation~\eqref{eq:green-form-convergence}
states explicitly the normalization required for the theorem.
The result concerns Gaussian laws, not convergence of the
Hadamard operators themselves, nor a general fractal scaling
limit.
\end{remark}

\subsection{Convergence of integrated Hadamard increments}
\label{SS:HadamardIncrements}

\begin{corollary}[Convergence of integrated Hadamard increments]
\label{cor:integrated-hadamard}
Let $0<s<t\leq1$, and choose nested square-lattice
approximations $U_{h,s}\subset U_{h,t}$ as in
the approximation convention above. Write
$\tilde{\mathbf G}_{h,s}$ and
$\tilde{\mathbf G}_{h,t}$ for the inverse combinatorial
Dirichlet Laplacians, extended by zero to the same lattice.
For $f,g\in C_c^\infty(\Omega_1)$, write
$f_h(x):=f(hx)$ and $g_h(x):=g(hx)$. Then
\[
\begin{aligned}
&\lim_{h\to0}h^4
 \sum_{x,y} f_h(x)
 \bigl(\tilde G_{h,t}(x,y)
       -\tilde G_{h,s}(x,y)\bigr)g_h(y)\\
&\qquad =
 \langle f,(G_t-G_s)g\rangle_{L^2(\Omega_1)},
\end{aligned}
\]
where $G_\tau=(-\Delta_{\Omega_\tau})^{-1}$ is extended
by zero to $L^2(\Omega_1)$. Moreover, the discrete
difference on the left is the sum of the positive Hadamard
increments $K_{h,j}K_{h,j}^*$ over the layers inserted
between $s$ and $t$.
\end{corollary}

\begin{proof}
Apply Lemma~\ref{lem:discrete-resolvent} separately to
$\Omega_s$ and $\Omega_t$ and subtract the two limiting
Green forms. The nested-set Hadamard factorization of
Lemma~\ref{lem:nested-hadamard} gives the telescoping identity
\[
\tilde{\mathbf G}_{h,t}
-\tilde{\mathbf G}_{h,s}
=\sum_{j\colon\,s< t_j\leq t}K_{h,j}K_{h,j}^*,
\]
where $t_j$ indexes the intermediate layers. This proves
both assertions.
\end{proof}

\begin{remark}
The corollary establishes convergence of the integrated
Green-operator variation in the sense of quadratic and
polarized forms. When the moving domains are sufficiently
smooth, the limiting increment also has the classical
Hadamard boundary-integral representation.
\end{remark}

\end{document}